\documentclass[a4paper,12pt]{article}

\usepackage[english]{babel}
\usepackage[utf8x]{inputenc}
\usepackage[T1]{fontenc}

\usepackage[a4paper,top=3cm,bottom=2cm,left=3cm,right=3cm,marginparwidth=1.75cm]{geometry}

\usepackage{cite}
\usepackage{amsmath,amssymb,amsthm}
\usepackage{mathtools}
\usepackage{graphicx}
\usepackage[colorinlistoftodos]{todonotes}
\usepackage[colorlinks=true, allcolors=blue]{hyperref}
\newtheorem{theorem}{Theorem}[section]%
\newtheorem{corollary}[theorem]{Corollary}%

\newtheorem{definition}[theorem]{Definition}%

\newcommand{\dE}{\mathbb{E}}

\title{A special case of the existential version of the Non-commutative Khintchine inequality}
\author{Satyaki Mukherjee}

\begin{document}
\maketitle
\section{Introduction}
 The following theorem is known as the Non-commutative Khintchine inequality.
 \begin{theorem}
 Let $A_1, ..., A_n$ be $d \times d$ symmetric matrices. Let $e_1, ..., e_n$ be random variable taking values $1$ or $-1$ with equal probability. Then there exists a constant $K$ such that $$\dE[||\sum_i e_iA_i||] \leq K\sqrt{\log d}\sqrt{||\sum_i A_i^2||}.$$ 
 \end{theorem}
 
 For further reading on this inequality we refer to the works of  \cite{TRO}, \cite{RB}, \cite{RLYP}. Given this inequality there is a question that arises out of similar inequalities in the vein of Spencer's six sigma theorem \cite{SP}. While there are logarithmic factors in the suprema of the random signing, is there a specific signing for which the logarithm is removed. In particular we ask the following question.
 
 Does there exist a constant, $K$ independent of $d$, such that given $A_1, ..., A_n$ symmetric $d \times d$ matrices, is it true that there exists a signing, as in a sequence of $1$ and $-1$, $e_1, ..., e_n$, such that $$\dE[||\sum_i e_iA_i||] \leq K \sqrt{||\sum_i A_i^2||}?$$
 Here we prove a special case of the above.
 \begin{theorem}
Let $A = \{a_{ij}\}_{i,j\in \mathbb{N}}$ be a bounded operator. Then there exists a signing of $A$ such that $$||A\circ  S||_2 < 2||A||_{l_\infty(l_2)},$$ where $A\circ S$ denotes the matrix generated by the entry-wise product of $A$ and $S$.
\end{theorem}

 A similar result was proved in 1997 by Fran{\c{c}}oise Lust-Piquard \cite{FL}.
 \begin{theorem}
 For every matrix $A=(a_{ij})$ such that $A$ and $A^*$ are bounded in $l^{\infty}(l^2)$ norm,
there exists a matrix $B=(b_{ij})$ defining a bounded operator: $l^2 \rightarrow l^2$ such that

(i) $|B|_{2\rightarrow 2} \leq K max\{|A|_{l^{\infty}(l^2)}, |A^*|_{l^{\infty}(l^2)}\}$

(ii) $\forall i,j \in \mathbb{N}$ , $|b_{ij}| \geq |a_{ij}|$,

where K is an absolute constant and $|A|_{l^{\infty}(l^2)} \coloneqq \max_{j} \sqrt{\sum_{i} a_{ij}^2}$ .
 \end{theorem}

Our theorem is an improvement of this result in two ways. Firstly we show that there exists a signing of the matrix $A$ which satisfies the above theorem. (A signing is a matrix $B$ such that $|b_{ij}| = |a_{ij}|$). Secondly we get that the constant $K$ as $\sqrt{2}$ suffices. In fact our constant is tight in the case of signings. Results in this vein but for different norms have been proved by Pisier \cite{Pi3}. In particular they prove that given a matrix $A$, there exists a signing, $B$ such that $$||B||_{\infty \rightarrow 1} \leq K ||A||_{l_1(l_2)}.$$

\section{Notation and definitions}
Given a $n \times n$ matrix $A = \{a_{ij}\}$, denote by $$|A|_2 = \max_{x\in \mathbb{R}^n} \frac{||Ax||_2}{||x||_2}.$$

And denote $$|A|_{l_\infty(l_2)} = \max_{j} \sqrt{\sum_{i} a_{ij}^2}.$$

\begin{definition}
Given two $n \times n$ matrices $A =(a_{ij})$ and $B=(b_{ij})$, define their Schur product, $A\circ B$ to be  the matrix whose $(i,j)$'th entry is $a_{ij}b_{ij}$.
\end{definition}
\begin{definition}\textbf{(Signings)}
 A sign matrix is a $n \times n$ matrix, $S$ all of whose entries are $1$ or $-1$. A symmetric sign matrix is, as the name suggests, a symmetric matrix which is also a sign matrix. Let $\mathcal{S}$ be the collection of all symmetric sign matrices of size $n$.
 \end{definition}

Given any matrix $A$ and a sign matrix $S$, a signing of $A$ by $S$ is simply the matrix $A\circ S$.

\begin{definition}
A dimer arrangement $D$ of size $d$ on the set $\{ 1, 2, ..., n\}$ is a set of tuples $\{(i_1,j_1), ... , (i_d,j_d)\}$ such that  all the the $i$'s and $j$'s are distinct from one another. The size of the dimer arrangement $D$, denoted by $|D|$ is the number of tuples, $d$. Let $\mathcal{D}$ be the set of all dimer arrangements of size $d$.
\end{definition}

The canonical weight of a dimer arrangement on a matrix $A$ is defined to be $W_A(D)= \Pi_{(i,j) \in D} a_{ij}$.

Again given a matrix $A$, define the dimer partition function as$$Z_d(A) = \sum_{|D| = d} W_A(D),$$ where the sum runs over all possible dimer arrangements of size $d$.

\begin{definition}Finally given a $n \times n$ matrix $A$, the matching polynomial of $A$ is then defined to be $$\mu_A (x) = \sum_{i = 0}^{n/2} (-1)^d Z_d(A) x^{n-2d}.$$
\end{definition}

\section{Preliminaries}
The following is a trivial modification of theorem 3.6 in \cite{IF1}.
\begin{theorem}
Let $A$ be a symmetric matrix. As previously defined let $\mathcal{S}$ be the set of all symmetric signing matrices. Let $S$ be a random signing chosen uniformly from $\mathcal{S}$.  Then $$\mathbb{E}_{S} [\mathrm{det}(xI-A\circ S)] = \mu_{A\circ A}(x).$$
\end{theorem}

\begin{proof}
Let $Sym(T)$ denote the set of permutations of a set $T$. Let $|\sigma|$ denote the entropy or the number of inversions of a permutation $\sigma$. Then \begin{align*}
\mathbb{E}_{S} [\det(xI-A\circ S)] &=  \mathbb{E}_{S} \left[ \sum_{\sigma \in Sym([n])} (-1)^{|\sigma|}\prod_{i=1}^{n} (xI - A\circ S)_{i,\sigma(i)} \right] \\
&= \mathbb{E}_{S} \left[ \sum_{k=0}^n x^{n-k} \sum_{T \subset [n]; |T| = k}\sum_{\sigma \in Sym(T)}(-1)^{|\sigma|} \prod_{i=1}^{k} (- A\circ S)_{i,\sigma(i)} \right] \\
&=   \sum_{k=0}^n x^{n-k} \sum_{T \subset [n]; |T| = k}\sum_{\sigma \in Sym(T)}(-1)^{|\sigma|} \mathbb{E}_S[\prod_{i=1}^{k} -a_{i,\sigma(i)}s_{i,\sigma(i)}] .
\end{align*}

But the $s_{i,j}$ are all independent excepting $s_{i,j} = s_{j,i}$, with expectation, $\dE(s_{i,j}) = 0$. Thus only even powers of $s_{i,j}$ survive the expectation. So we may only consider permutations which only have orbits of size 2. These are just the perfect matchings on $S$ or alternatively exactly all the dimer arrangements of size $|S|$. There are no such matchings when $|S|$ is odd. Otherwise its entropy is $|S|/2$. And since $$\mathbb{E}[(-a_{i,j}s_{i,j})^2] = a_{i,j}^2,$$ we get $$\mathbb{E}_S[\det(xI-A\circ S)] = \sum_{k=0}^{n/2} x^{n-2k} \sum_{|D| = k; D \in \mathcal{D}} (-1)^{k} \prod_{(i,j) \in D} a_{i,j}^2 = \mu_{A\circ A}(x).$$
\end{proof}

The next theorem is the famous Heilman-Leib theorem which proves that the matching polynomial is real rooted and  gives a bound for the maximum root of the matching polynomial of a matrix. It can be found in \cite{HL1} as theorem 4.2 and 4.3. 

\begin{theorem}
Let $A$ be a  symmetric matrix with real positive entries. Let $b$ be the maximum row sum of $A$ i.e. $b = \max_{i \in [n]} \{\sum_j a_{i,j}\}$. Then $\mu_A(x)$ is real rooted and any root $\lambda$ satisfies, $\lambda < 2\sqrt{b}$.
\end{theorem}

An immediate corollary of the above theorem is the following.
\begin{corollary} Let $A$ be any symmetric matrix. Let $r_1, ..., r_n$ be the rows of $A$. Let $||r_i||_2$ be the L2 norm of the vector $r_i$. Let $|A|_{l_\infty(l_2)} = \max_{i \in [n]} ||r_i||_2  = \max \frac{||Ax||_{\infty} }{||Ax||_2}$.

Then every root $\lambda$ of $\mu_{A\circ A} (x)$ satisfies $|\lambda| < 2|A|_{l_\infty(l_2)}.$
\end{corollary}

The final piece of the puzzle is the theory of interlacing families found in both \cite{HL1} and \cite{IF1}.

\begin{definition}We say a polynomial $g(x) = \prod_{i \in [n-1]} (x-\alpha_i)$ interlaces $f(x) = \prod_{i \in [n]} (x-\beta_i)$ iff $\beta_1 \leq \alpha_1 \leq \beta_2 .... \leq \alpha_{n-1} \leq \beta_n$. \end{definition}

We say that the polynomials $f_1,... ,f_k$ have a common interlacing if there is a polynomial $g$ such that $g$ interlaces $f_i$ for each $i$. 

It turns out that when all the polynomials are monic and of same degree, they have a common interlacing if and only if every convex combination is real-rooted.

\begin{definition}
Let $S_1, ..., S_m$ be finite sets and for every assignment $s_1, ..., s_m \in S_1 \times .... \times S_m$, let $f_{s_1, ..., s_m}$ be  a real rooted n degree polynomial with positive leading coefficient. For a partial assignment  $s_1, ..., s_k \in S_1 \times .... \times S_k$, with $k<m$ define $$f_{s_1,...,s_k} = \sum_{s_{k+1} \in S_{k+1}; ...; s_m \in S_m} f_{s_1, ..., s_m},$$ as well as $$f_{\emptyset} =  \sum_{s_{1} \in S_{1}; ...; s_m \in S_m} f_{s_1, ..., s_m}.$$ We say that $\{f_{s_1}, ..., f_{s_m}\}_{S_1, ..., S_m}$ form an interlacing family, iff for all $k<m$, and for all $s_1, ..., s_k) \in S_1 \times ... \times S_k$  the set of polynomials $\{f_{s_1,...,s_k,t}\}_{t \in S_{k+1}}$ have a common interlacing.
\end{definition}
 
Then we have the following theorems from \cite{IF1} (thm 4.4)
\begin{theorem}
Let $S_1, ..., S_m$ be finite sets and let $\{f_{s_1, ..., s_m}\}$ be an interlacing family of polynomials. Then there exists some $s_1, ..., s_m \in S_1 \times ... \times S_m$ such that the largest root of $f_{s_1, ..., s_m}$ is less than or equal to the largest root of $f_{\emptyset}$.
\end{theorem}

Finally let $S_i = \{1,-1\}$ for $1\leq i \leq m=n(n+1)/2$. Then we note that each element of $S_1 \times ... \times S_m$ corresponds to a symmetric signed matrix $S_{s_1, ..., s_m}$. Thus we can define the polynomial $f_{s_1, ..., s_m}(x) = \det(xI- A\circ S_{s_1, ..., s_m})$. Again by \cite{IF1} (Theorem 5.2), we get that 

\begin{theorem}
$f_{s_1, ..., s_m}(x)$ forms an interlacing family.
\end{theorem}

While in the referred paper the authors use this theorem only on adjacency matrices of graphs (whose entries are only $0$ or $1$), its proof is valid over any symmetric matrix. The key idea behind the proof is that there is this class of functions on matrices called determinant-like, which remain determinant-like (and real-rooted) under a rank-one update.

\section{Statement and Proof of main Theorem}
Now we proceed to proving our main theorem.
\begin{theorem}
Let $A$ be any $n \times n$ matrix. Then there exists a signing matrix not necessarily symmetric such that $$||A\circ S||_2 \leq 2 ||A||_{l_\infty(l_2)}.$$
\end{theorem}

\begin{proof}
Given $A$, define the dilation $A_D$ to be the $2n \times 2n$ matrix, $$A_D = \begin{bmatrix}
0 & A \\
A^T & 0
\end{bmatrix},$$ where $A^T$ denotes the transpose of $A$.

Note that $A_D$ is a symmetric matrix. Let $\mathcal{S}$ be the set of all $2n \times 2n$ sign matrices. Let $S$ be a sign matrix chosen uniformly from $\mathcal{S}$.

Then by Theorem 3.1, $$\mathbb{E}_{S} [\det(xI-A_{D}\circ S)] = \mu_{A_D\circ A_D}(x).$$

But by Theorem 3.7, the polynomials in the left hand side of the above equation form an interlacing family. Therefore by Theorem 3.6, there exists some signing matrix $S'$ such that, the largest root of $\det(xI-A_D\circ S')$ is less than or equal to the largest root of $\mu_{A_D\circ A_D}(x)$.

But using Corollary 3.3, every root of $\mu_{A_D\circ A_D}(x)$ is in modulus smaller than $2|A_D|_{l_\infty(l_2)}$.

Combining all this we have a $2n \times 2n$ sign matrix $S'$ such that the largest eigenvalue of $A_D\circ S'$ is less $2|A_D|_{l_\infty(l_2)}$. 
Let $S' = \begin{bmatrix}
S_1 & S_2 \\
S_2^T & S_4
\end{bmatrix}$
Then using Schur complements, $$\det(xI - A_D\circ S) = x^n\det(x - x^{-1}(A\circ S_2)(A^T\circ S_2^T)) =  \det(x^2 - (A\circ S_2)(A\circ S_2)^T).$$ Thus the largest eigenvalue of $A_D\circ S'$  is simply the largest singular value or the L2 norm of $A\circ S_2$.

So we have a signing matrix $S_2$, with $$||A\circ S_2||_2 < 2 ||A\circ S_2||_{l_\infty(l_2)} = ||A||_{l_\infty(l_2)}.$$
\end{proof}

\begin{theorem}
(Extension to infinite dimensions). Let $A = \{a_{ij}\}_{i,j\in \mathbb{N}}$ be a bounded infinite dimensional operator. Then there exists a signing of $A$ such that $$||A\circ  S||_2 < 2||A\circ S||_{l_\infty(l_2)}.$$
\end{theorem}

\begin{proof}
For any integer $n$, let $A_n$ be the operator constructed from $A$ by taking the upper $n \times n$ part of $A$ and filling everything else with $0$. Then by our previous result, there exists a signing $S_n$ such that $$||A_n\circ  S_n||_2 < 2||A_n\circ S_n||_{l_\infty(l_2)} = 2||A_n||_{l_\infty(l_2)} \leq 2||A||_{l_\infty(l_2)}.$$

Thus as the sequence $\{A_n \circ S_n\}$ is uniformly bounded, by using sequential Banach Alaoglu, there is a subsequence $k_n$ such that $A_{k_n} \circ S_{k_n}$ converges weakly to some matrix $B$. Note that $k_n$ approaches infinity, thus eventually every $i,j$ position of this subsequence is either $a_{ij}$ or $-a_{ij}$. Thus the weak limit is also a signing of $A$. Denote $B_n = A_{k_n} \circ S_{k_n}$.

Thus $B_n^*B_n$ also converges weakly to $B^*B$. Then for any $x$, we have that $\langle x, B_n^*B_n x\rangle$ converges to $\langle x, B^*Bx\rangle = ||Bx||_2$. Thus we have that for any $x$, such that $||x||_2 = 1$, $$||Bx||_2 < 2||A||_{2,\infty}.$$
\end{proof}

\section*{Acknowledgements}
I would like to thank my advisor Nikhil Srivastava and Ramon Van Handel for mentioning this problem and the helpful discussions regarding it.

\bibliographystyle{amsalpha}
\bibliography{ref}
\end{document}